\documentclass[11pt]{article}

\usepackage{amssymb,amsfonts,amsthm,amsmath,mathrsfs}
\usepackage{graphicx}
\usepackage[margin= 1 in]{geometry}

  \newcommand{\Z}{\ensuremath{\mathbb{Z}}}%
  \newcommand{\N}{\ensuremath{\mathbb{N}}}%

    \newcommand{\A}{\ensuremath{\mathcal{A}}}%

        \newcommand{\G}{\ensuremath{\mathcal{G}}}%

        \newcommand{\M}{\ensuremath{\mathcal{M}}}%

        \renewcommand{\H}{\ensuremath{\mathcal{H}}}%

  \newcommand{\supp}{\ensuremath{\textrm{supp}}}%
    \newcommand{\acts}{\ensuremath{\curvearrowright}}%
  \newcommand{\sub}{\ensuremath{\operatorname{Sub}}}%
    \newcommand{\aut}{\ensuremath{\operatorname{Aut}}}%

  \newcommand{\stab}{\ensuremath{\operatorname{Stab}}}%
  \newcommand{\homeo}{\ensuremath{\operatorname{Homeo}}}%

      \newcommand{\oGamma}{\ensuremath{\overline{\Gamma}}}%

\theoremstyle{definition}
  \newtheorem{defin}{Definition}[section]

\theoremstyle{plain}
  \newtheorem{thm}[defin]{Theorem}
  \newtheorem{main thm}{Theorem}
  \newtheorem{prop}[defin]{Proposition}
  
  \newtheorem{cor}[defin]{Corollary}
  \newtheorem{lemma}[defin]{Lemma}

\theoremstyle{remark}
  \newtheorem{remark}[defin]{Remark}

  \newtheorem{example}[defin]{Example}
  
  \author{Nicol\'as Matte Bon}
  \date{February 2016}
  
  \title{Full groups of bounded automaton groups}

\begin{document}
  \maketitle\abstract{We show that every bounded automaton group can be embedded in a finitely generated, simple amenable group. The proof is based on the study of the topological full groups associated to the  Schreier dynamical system of the mother groups. We also show that if $\G$ is a minimal \'etale groupoid with unit space the Cantor set, the group $[[\G]]_t$ generated by all torsion elements in the topological full group has simple commutator subgroup. }  
  
\section{Introduction}
Groups generated by finite automata are a classical source of groups acting on rooted trees. A well-study family of automaton groups are groups generated by bounded automata. Examples of such groups are the Grigorchuk group of intermediate growth~\cite{Grigorchuk:grigorchukgroups}, Gupta-Sidki groups~\cite{Gupta-Sidki}, the Basilica group~\cite{Grigorchuk-Zuk:basilica}, iterated monodromy groups of post-critically finite polynomials~\cite{Nekrashevych:book}. 
One feature of automata groups is that they provide many examples of amenable groups \cite{Grigorchuk:grigorchukgroups, Bartholdi-Virag:amenability, Bartholdi-Kaimanovich-Nekrashevych, Amir-Angel-Virag:linear, Juschenko-Nekrashevych-Salle:recurrentgroupoids}, that are often non-elementary amenable~\cite{Grigorchuk:grigorchukgroups, Grigorchuk-Zuk:basilica, Juschenko:nonelementary}. In particular, every bounded automaton group is amenable by a theorem of Bartholdi, Kaimanovich and Nekrashevych~\cite{Bartholdi-Kaimanovich-Nekrashevych}.

Another notion that has attracted some attention in relation to amenability is the \emph{topological full group} of  a group of homeomorphisms. Let $G\acts X$ be a countable group acting by homeomorphisms on the Cantor set. The topological full group of  $G\acts X$ is the group $[[G]]$ of all homeomorphisms of $X$ that are locally given by the action of elements of $G$. More generally a topological full group can be defined for every \emph{\'etale groupoid} $\G$  ~\cite{Matui:groupoids}  (these definitions are recalled in Section \ref{S: groupoids}). This recovers the previous definition if $\G$ is the groupoid of germs of the action $G\acts X$. Topological full groups were first defined and studied by Giordano, Putnam and Skau in~\cite{Giordano-Putnam-Skau:flipconjugacy} in the case of minimal $\Z$-actions on the Cantor set.  The commutator subgroups of topological full groups of minimal $\Z$-subhifts provided the first examples of finitely generated infinite simple groups that are amenable (their simplicity and finite generation was established by Matui \cite{Matui:simple} and their amenability by Juschenko and Monod~\cite{Juschenko-Monod:simpleamenable}.  Amenability of some topological full groups that are not given by $\Z$-actions was established in \cite{Juschenko-Nekrashevych-Salle:recurrentgroupoids, ExtAmen}. Recently new examples were given by Nekrashevych  in \cite{Nekrashevych:periodic}, who obtained the first examples of infinite, finitely generated, simple periodic groups that are amenable.\\

Many  connections between the theory of automata groups and topological full groups exist. This was first noticed by Nekrashevych in~\cite{Nekrashevych:inversesemigroups} and became more apparent recently, see ~\cite{Juschenko-Nekrashevych-Salle:recurrentgroupoids, GrigorchukTopFull}. One motivation for this paper is to further explore this connection.

 Let $G$ be a group acting level-transitively on a regular rooted tree $T_d$.  It is well-known that the topological full group of the action $G\acts \partial T_d$ on the boundary of the rooted tree does not provide any substantially new example of finitely generated groups (more precisely, all its finitely generated subgroups embed in a finite extension of a direct power of $G$).

We consider instead the topological full group of the \emph{Schreier dynamical system} of $G$ in the sense of Grigorchuk. In the terminology of Glasner and Weiss  \cite{Glasner-Weiss:URS}, this can be defined as  the  \emph{uniformly recurrent subgroup} (URS) arising from the \emph{stability system} of  the action on the boundary of the rooted tree $\partial T_d$.  Namely, consider the stabiliser map 
\[\stab\colon \partial T_d\to \sub(G)\]
where $\sub(G)$ is the space of subgroups of $G$, endowed with the Chabauty topology. This map is not, in general, continuous.  Let $Y\subset \partial T_d$ be the set of continuity points of $\stab$ (which is always a $G_\delta$-dense subset of $\partial T_d$), and let $X\subset \sub{(G)}$ be the closure of the image of $Y$.  Then $G$  acts continuously on $X$ by conjugation. The action $G\acts X$ is called the \emph{Schreier dynamical system}Ê of $G$.

We define a full group $[[G]]$ associated to $G$ as the topological full group of the Schreier dynamical system $G\acts X$. 

As a consequence of the results from Sections \ref{S: amenability} and \ref{S: finite generation}, we have:

\begin{thm}\label{T: main}
Every group generated by a bounded activity automaton can be embedded in a finitely generated, simple amenable group.
\end{thm}

The proof is based on a detailed study of $[[G]]$ when  $G$ is one of the bounded activity \emph{mother groups},  a family of bounded automaton groups introduced in  \cite{Bartholdi-Kaimanovich-Nekrashevych, Amir-Angel-Virag:linear} that contain all other bounded automaton groups as subgroups. We show that $[[G]]$ is amenable and admits a finitely generated, infinite simple subgroup that contains $G$. The preliminary embedding in the mother group cannot be avoided to obtain a finitely generated group, as finite generation does not hold for some different choices of $G$ (see Lemma \ref{L: basilica}).\\

We also study simplicity of subgroups of topological full groups for more general group actions and \'etale groupoids.  Matui proves in \cite{Matui:simple, Matui:groupoids}Ê  that, for some classes of group actions and groupoids,   the topological full group has simple commutator subgroup (in particular this holds for any minimal $\Z^d$-action on the Cantor set, and more generally for every \emph{almost finite}Ê and \emph{purely infinite} Êminimal \'etale groupoid, see \cite{Matui:groupoids}). It is not known whether this holds true for every minimal action of a countable group on the Cantor set (more generally for every minimal \'etale groupoid with Cantor set unit space).

Given an \'etale groupoid $\G$ with Cantor set unit space (e.g. the groupoid of germs of a group action on the Cantor set $G\acts X$), we denote by $[[\G]]_t$ the subgroup of the topological full group generated by torsion elements. We show the following result:

\begin{thm}[Theorem \ref{T: simplicity}] \label{T: main 2}
Let $\G$ be a minimal \'etale groupoid with unit space the Cantor set. Then $[[\G]]_t'$ is simple.
\end{thm}
Here $[[\G]]_t'$ denotes the derived subgroup of $[[\G]]_t$. It is not known whether there exists $\G$ as in the statement, for which the inclusion $[[\G]]_t'\le[[\G]]'$ is strict. 

A very similar result has been recently shown by Nekrashevych in \cite{Nekrashevych:groupoidsimple}, and appeared while the writing of this paper was being completed\footnote{In a preliminary version of this paper, Theorem \ref{T: main 2} was proven under an additional assumption on $\G$, that was removed in the present version.}Ê. He defines a subgroup $A(\G)$ of $[[\G]]$ analogous to the alternating group, and shows that $A(\G)$ is simple if $\G$ is minimal,  and that $A(\G)$ is finitely generated if $\G$ is expansive.  We refer the reader to \cite{Nekrashevych:groupoidsimple}Ê for the definition of $A(\G)$. Since it is apparent from the definitions that $A(\G)\unlhd [[\G]]_t'$, it follows from Theorem \ref{T: main 2} that $A(\G)= [[\G]]_t'$ if $\G$ is minimal.

 This paper is structured as follows. In Section 2 we recall preliminaries on \'etale groupoids and their topological full groups, and prove Theorem \ref{T: main 2}. In Section 3 we recall preliminaries on groups acting on rooted trees and their Schreier dynamical systems. In Section 4 we study the Schreier dynamical system of the alternating mother groups $M\acts X$. We show that this action can be efficiently encoded by a Bratteli diagram representation. Combined with a result from \cite{Juschenko-Nekrashevych-Salle:recurrentgroupoids} Êthis allows to show that $[[M]]$ is amenable. Finally in Section 5 we show that $[[M]]_t'$ is finitely generated.

   \paragraph*{Acknowledgements}
   I am grateful to Anna Erschler for many discussions, and to G. Amir for a very useful conversation. I thank R. Grigorchuk, H. Matui and V. Nekrashevych for useful comments on a previous version. Part of this work was completed at the Bernoulli Center during the trimester ``Analytic and Geometric Aspects of Probability on Graphs'', the author acknowledges the CIB and the organizers of the program.
\section{\'Etale groupoids and topological full groups} \label{S: groupoids}
\subsection{Preliminary notions}
A \emph{groupoid} $\G$ is a small category (more precisely, its set of morphisms) in which every morphism is an isomorphism. Recall that a category is said to be \emph{small}Ê if the collection of its objects and morphisms are sets. 
The set of objects of a groupoid is called its \emph{unit space}.

Throughout the section let $\G$ be a groupoid and $X$ be its unit space.
For every $\gamma \in \G$ the initial and final object of $\gamma$ are denoted $s(\gamma)$ and $r(\gamma)$. The maps $s, r: \G\to X$ are called the \emph{source} and the \emph{range} maps. Thus, the product of $\gamma, \delta \in \G$ is defined if and only if $r(\delta)=s(\gamma)$. In this case we have $s(\gamma\delta)=s(\delta)$ and $r(\gamma\delta)= r(\gamma)$. We will systematically identify $X$ with a subset of $\G$ via the identification $x\mapsto \operatorname{Id}_x$ (the identity isomorphism of the object $x$).

An \emph{\'etale groupoid} is a  groupoid $\G$ endowed with a second countable locally compact groupoid topology so that the range and source maps $r,s\colon \G\to X$ are open local homeomorphism. Note that we do not require $\G$ to be Hausdorff.

\begin{example}\label{E: action groupoid}
Let $G$ be a countable group acting on a topological space $X$ by homeomorphisms. The \emph{action groupoid} is the groupoid $\G=G\times X$. The product of $(g,x)$ and $(h,y)$ is defined if and only if $hy=x$ and in this case $(g, x)(h,y)=(gh, y)$. The unit space of $\G$ naturally identifies with $X$. The range and source maps are given by $s(g, x)=x$ and $r(g, x)=gx$. The topology on $\G$ is the product topology on $G\times X$, where $G$ is endowed with the discrete topology. This topology makes $\G$ a Hausdorff \'etale groupoid.
\end{example}

\begin{example}\label{E: groupoid of germs}
Let again $G$ be a countable group acting on a topological space $X$ by homeomorphisms. Given $g\in G$ and $x\in X$ the \emph{germ} of $g$ at $x$ is denoted $\operatorname{germ}(g, x)$.  Germs of elements of $G$ form a groupoid $\G$, the \emph{groupoid of germs} of the action. The groupoid of germs is naturally endowed with a topology where a basis for open sets is given by sets of the form 
\[ U(g, V)= \{ \operatorname{germ}(g, x)\: : x\in V\},\]
where $g\in G$ is fixed and $V\subset X$ is an open subset. This topology makes $\G$ an \'etale groupoid, which is non-Hausdorff in many interesting cases. 

\end{example}

\begin{example}\label{E: pseudogroup groupoid}
More generally let $\mathscr{F}$ be a \emph{pseudogroup} of \emph{partial homeomorphisms} of a topological space $X$. A partial homeomorphism of $X$ is a homeomorphisms between open subsets of $X$ 
\[\tau\colon U\to V,\quad U, V\subset X \text{ open }\] 
where $U$ and $V$ are called the \emph{domain} and the \emph{range} of $\tau$.
A pseudogroup $\mathscr{F}$ is a collection of partial homeomorphisms which is closed under taking inverses, taking restriction to an open subset of $U$, taking composition on the locus where it is defined, and verifying the following: 
whenever $\tau$ is a partial homeomorphism of $X$ as above and the domain $U$ admits a covering by open subsets $U=\cup_{i\in I} U_i$ so that $\tau|_{U_i}\in \mathscr{F}$ for every $i\in I$ then $\tau \in \mathscr{F}$.
One can naturally associate to $\mathscr{F}$ a \emph{groupoid of germs}  $\G$, and endow it with a topology in the same way as in Example \ref{E: groupoid of germs}, which makes it an \'etale groupoid.
\end{example}

Two  \'etale groupoids are said to be \emph{isomorphic} if they are isomorphic as topological groupoids.

An \'etale groupoid is said to be \emph{minimal} if it acts minimally on its unit space (i.e. every orbit is dense in $X$).

Let $\G$ be an \'etale groupoid. A \emph{bisection} is an open subset $\mathcal U\subset \G$ so that the source and range maps $s\colon\mathcal U\to s(\mathcal U)$ and $r\colon \mathcal U\to r(\mathcal U)$ are homeomorphism onto their image. To any bisection $\mathcal U$ one can associate a partial homeomorphism of $X$ (i.e. a homeomorphism between open subsets of $X$) given by

\[\tau_{\mathcal{U}}:=r\circ s^{-1}\colon s(\mathcal U)\to r(\mathcal U).\]

A bisection $\mathcal U$ is said to be \emph{full} if $s(\mathcal U)=r(\mathcal U)=X$.

The \emph{topological full group} of an \'etale groupoid $\G$ is the group 
\[ [[\G]]=\{\tau_{\mathcal U}\: : \: \mathcal U\subset \G\: \text{ full bisection }\}\le \homeo(X).\]

\begin{example}
Let $\G$ be either an action groupoid as in Example \ref{E: action groupoid}, or a groupoid of germs as in Example \ref{E: groupoid of germs}. Then $[[\G]]$ coincides with the group of homeomorphisms $h$  of $X$ the following property: for every $x\in X$ there exists a neighbourhood $V$ of $x$ and an element $g\in G$ so that $h|_V=g|_V$. 

\end{example}

\begin{example}
Let $\G$ be the groupoid of germs associated to a pseudogroup $\mathscr{F}$ as in Example \ref{E: pseudogroup groupoid}. Then $[[\G]]$ coincides with the set of all elements of $\mathscr{F}$ whose domain and range are equal to $X$.
\end{example}

For a groupoid $\G$ with unit space $X$ the set $\G_x=\{\gamma \in \G\: : \: s(\gamma)=r(\gamma)=x\}$ forms a group, the \emph{isotropy group} at the point $x\in X$. A groupoid is said to be \emph{principal} (or an \emph{equivalence relation}) if $\G_x=\{\operatorname{Id}_x\}$ for every $x\in X$.

From now on, all the groupoids that we consider have a unit space $X$ homeomorphic to the Cantor set.\\

An \emph{elementary subgroupoid}  $\mathcal{K}\le \G$ is a principal compact open subgroupoid with unit space $X$. 

An \'etale groupoid is said to be an \emph{AF-groupoid}  if it is a countable union of elementary subgroupoids. In particular, AF-groupoids are principal. 

We now recall basic concepts concerning \emph{Bratteli diagrams} and their relation to AF-groupoids.

A \emph{Bratteli diagram} is a labelled directed graph $B$, possibly with multiple edges, whose vertex set $V$ is a disjoint union of finite \emph{levels} $V_n$, $n\ge 0$. The level $V_0$ contains only one vertex, called the \emph{top vertex}. Every edge $e$ connects a vertex in a level $V_n$ to a vertex in the next level $V_{n+1}$ for some $n\ge 0$. The starting vertex and the ending vertex of an edge $e$ are called the \emph{origin} and \emph{target} of $e$ and are denoted $o(e)$ and $t(e)$ respectively. We denote $E_n$ the set of edges going from vertices in $V_n$ to vertices in $V_{n+1}$. 

The \emph{path space} of a Bratteli diagram $B$ is the set of infinite directed paths in $B$, i.e.
\[X_B=\{e_0e_1e_2\cdots \: : \: e_i\in E_i, \: t(e_i)=o(e_{i+1})\}.\]
It is endowed with the topology induced by the product of the discrete topology on each $E_i$.

A  Bratteli diagram is said to be \emph{simple} if for every $n$ and every vertex $v\in V_n$ there exists $m>n$ so that $v$ is connected by a path to every vertex of $V_m$.

To a Bratteli diagram one can associate an AF-groupoid $\H_B$, the \emph{tail groupoid}, as follows.

Let $\gamma=f_0f_1\cdots f_n$ be a path in $B$ starting at the top vertex, i.e. $f_i\in E_i$ and $t(f_i)=o(f_{i+1})$. We denote $t(\gamma)=t(f_n)$.  The sets of infinite paths in $X_B$ starting with $\gamma$ is called a \emph{cylinder subset} of $X_B$ and is denoted $C_{\gamma}$. Clearly cylinders subsets are clopen and form a basis for the topology on $X_B$.

Given $v\in V_n$, the \emph{tower} corresponding to $v$, denoted $T_v$, is the collection of all cylinder subsets $C_\gamma$ so that $t(\gamma)=v$.

Let $C_\gamma, C_{\gamma'}$ be two cylinders in the same tower. Then there is a partial homeomorphism of $X_B$ with domain $C_\gamma$ and range $C_{\gamma'}$ given by 
\begin{align*}
\tau_{\gamma, \gamma'}\colon C_\gamma&\to C_{\gamma'}\\
 \gamma e_{n+1}e_{n+2}\cdots &\mapsto \gamma'e_{n+1}e_{n+2}\cdots.
\end{align*}

Let $\mathscr{F}_B$ be the pseudogroup generated by all partial homeomorphisms of this form. The groupoid of germs of this pseudogroup, endowed with a topology as in Example \ref{E: pseudogroup groupoid}, is called the \emph{tail groupoid} of $B$ and it is denoted $\H_B$.  
The tail groupoid is an increasing union of the subgroupoid $\H_B^{(n)}$ consisting of all germs of partial homeomorphisms of the form $\tau_{\gamma, \gamma'}$ where $\gamma$ and $\gamma'$ have length at most $n$. It is easy to check that each $\H^{(n)}_B$ is an elementary subgroupoid in $\H_B$. In particular $\H_B$ is an AF-groupoid. It is not difficult to check that the groupoid $\H_B$ is minimal if and only if the diagram $B$ is simple.

The groups $H_n=[[\H_B^{(n)}]]$ are finite and are isomorphic to a direct product of symmetric groups. The group $[[\H]]'$  is simple if and only if $B$ is simple, see \cite{Matui:simple}.

The following fundamental result says that every minimal AF-groupoid arises in this way.

\begin{thm}[Giordano--Putnam--Skau \cite{Giordano-Putnam-Skau:affable}]\label{T: AF equivalences}
Let $\G$ be a minimal \'etale groupoid with unit space the Cantor set. The following are equivalent.
\begin{itemize}
\item[(i)] $\G$ is an AF-groupoid;
\item[(ii)] there exists a simple Bratteli diagram $B$ so that $\G$ is isomorphic to $\H_B$.

\end{itemize}
Moreover in this case the topological full group $[[\G]]$ is locally finite.
\end{thm}

\subsection{A characteristic simple subgroup of topological full groups}
Given an \'etale groupoid $\G$, we denote by $[[\G]]_t< [[\G]]$ the subgroup generated by torsion elements, and by $[[\G]]'_t$ the derived subgroup of $[[\G]]_t$.

\begin{thm}\label{T: simplicity}
Let $\G$ be a minimal groupoid with compact totally disconnected unit space $X$. Then $[[\G]]_t'$ is simple.
\end{thm}

The reader may compare the present proof with the proof of Bezuglyi and Medynets \cite{Bezuglyi-Medynets:flipconjugacy} of simplicity of $[[\G]]'$ in the special case where $\G$ is the groupoid of germs of a minimal $\Z$-action on the Cantor set (this result was first shown by Matui \cite{Matui:simple} with a rather different proof).

We use the notations $[g, h]=ghg^{-1}h^{-1}$ or the commutator and $g^h=hgh^{-1}$ for the conjugation.

Let us fix some notation and terminology in the following definition.
\begin{defin}\label{D: infinitesimally generated}
Let $H$ be a group of homeomorphisms of a topological space $X$. Given two subsets $U, V\subset X$ we write $U\preceq_H V$ if there exists $h\in H$ such that $h(U)\subset V$. 

We say that $H$ is \emph{infinitesimally generated} if for every open subset $U\subset X$ the set 
\[ S_U=\{ h\in H : \supp(h)\preceq_H U\}\]
generates $H$.

We say that $H$ is \emph{doubly infinitesimally generated} if the following holds. For every open subset $U\subset X$ and every $g, h\in H$ there exist $g_1, \ldots g_m, h_1, \ldots h_n\in H$ such that $g=g_1\cdots g_m$, $h=h_1\cdots h_n$, and the condition 
\[\supp(g_i)\cup \supp(h_j)\preceq_H U\] 
holds for every $1\leq i \leq m$ and $1\leq j\leq n$.
\end{defin}

\begin{remark}\label{R: generating set infinitesimally generated}
To check that $H$ is doubly infinitesimally generated, it is enough to check that the condition above is satisfied for $g, h$ in a generating set of $H$.
\end{remark}

The idea in the proof of the following proposition is classical and has been used in many proofs of simplicity of groups (cf. in particular the argument at the end of the proof of \cite[Theorem 3.4]{Bezuglyi-Medynets:flipconjugacy}.)
\begin{prop}\label{P: normal contains commutator}
Let $H$ a group of homeomorphism of a topological Haurdoff space  $X$. Assume that $H$ is doubly infinitesimally generated. Then every non-trivial normal subgroup of $H$ contains the derived subgroup $H'$. In particular, if $H$ is perfect then it is simple.
\end{prop}

\begin{proof}

Let $N\unlhd H$ be a non-trivial normal subgroup. Let $g, h\in H$. We show that $[g, h]\in N$. Let $f\in N$ be non-trivial. Let $U\subset X$ be an open set so that $f(U)\cap U=\varnothing$ (here we use the assumption that $X$ is Hausdorff). Write $g=g_1\cdots g_m$ and $h=h_1\cdots h_n$ as in Definition \ref{D: infinitesimally generated} (with respect to the clopen subset $U$). The commutator $[g, h]$ belongs to the subgroup normally generated by $[g_i, h_j]$ for all $i, j$. Hence it is enough to show that $[g_i, h_j]\in N$ for every $i$ and $j$.  Let $k\in H$ be such that $k(\supp(g_j)\cup \supp(h_j))\subset U$. Since $k$ normalizes $N$ it is sufficient to show that $[g_i^k, h_j^k]\in N$. Hence, up to replacing $g_i, h_j$ with $g_i^k, h_j^k$ we may assume that $\supp(g_i)\cup\supp
(h_j)\subset U$. After this assumption is made, we have that $g_i^f$ commutes with $g_i$ and $h_j$ since its support is contained in $f(U)$. 
Since $f\in N$ we have $[g_i, f]\in N$. Using that $g_i^f$ commutes with $g_i$ and $h_j$ this implies
\[[h_j,g_i]=[h_j, g_i(g_i^{-1})^f]=[h_j, [g_i, f]]\in N\]
thereby concluding the proof.\qedhere

\end{proof}

\begin{lemma}\label{L: infinitesimally generated}
Let $\G$ be as in Theorem \ref{T: simplicity}. Then $[[\G]]_t$ is doubly infinitesimally generated.
\end{lemma}

\begin{proof}
Let us first show that $[[\G]]_t$ is infinitesimally generated.  Let $U\subset X$  be clopen. It is enough to show that any torsion element $g\in [[\G]]_t$ can be written as a product of elements in $S_U$. Let $d$ be the order of $g$. Pick a point $x\in X$ and enumerate its $g$-orbit by $x=x_1, x_2 \ldots x_{l}$ (for some $l| d$). By minimality one can find $y_1, \ldots y_{l}\in U$ lying in the same $\G$-orbit of $x$, and such that $x_1\ldots x_{l}, y_1, \ldots y_{l}$ are all distinct. Clearly there exists an element $k\in [[\G]]_t$ of order 2 such that $k(x_i)=y_i$ for every $i=1, \ldots l$. To see this, consider $\gamma_i \in \G$  such that $s(\gamma_i)= x_i$ and $r(\gamma_i)=y_i$ for every   $i=1, \ldots l$. For each $i$ let $\mathcal{U}_i\subset \G$ be a bisection containing $\gamma_i$ such that $x_i\in s(\mathcal{U}_i)$, $y_i\in r(\mathcal{U}_i)$ are clopen sets and are small enough so that the sets $s(\mathcal{U}_1),\ldots, s(\mathcal{U}_l), r(\mathcal{U}_1), \ldots, r(\mathcal{U}_l)$ are all disjoint.  Let $k\in [[\G]]_t$ be the element that  coincides with $\tau_{\mathcal{U}_i}$ on $s(\mathcal{U}_i)$ and with $\tau_{\mathcal{U}_i}^{-1}$ on $r(\mathcal{U}_i)$  for all $i$ and with the identity elsewhere. More formally $k=\tau_{\mathcal{V}}$ for the full bisection $\mathcal V= \mathcal U_1\cup\ldots \mathcal U_l\cup \mathcal U_1^{-1}\cup \cdots \mathcal{U}_l^{-1}\cup X'$ where $X'=X\setminus(\bigcup s(\mathcal{U}_i) \cup \bigcup r(\mathcal{U}_i)$).  

Now let $W$ be a clopen neighbourhood of $x$ and set $V=\cup_{i=0}^{d-1} g^i(W)$. Then $V$ is $g$-invariant, and we have $k(V)\subset U$ if $W$ is small enough.

By compactness we have proven that there exists a finite covering of $X$ by $g$-invariant clopen sets $V$ such that $V\preceq_{[[\G]]_t} U$. Up to taking a refinement we may assume that this covering is a partition (since taking intersections and differences preserves the $g$-invariance). Then $g=g_1\cdots g_m$ where each $g_i$ coincides with the restriction of $g$ on each element of the partition and with the identity elsewhere, hence $g_i\in S_U$. This proves that $[[\G]]_t$ is infinitesimally generated.

Now observe that the construction above yields the following more precise conclusion. For every torsion element $g\in [[\G]]_t$ and every clopen set $U\subset X$ there exist a writing  $g=g_1\cdots g_m$,  a partition into clopen subsets $X=V_1\sqcup \ldots\sqcup V_m$ and elements $k_1,\ldots k_m\in [[\G]]_t$ such that for every $i=1,\ldots, m$ we have
\begin{enumerate} \item $\supp(g_i)\subset V_i$;
\item the elements $k_i$ have order 2, $k_i(V_i)\subset U$ and moreover $k_i(V_i)\cap V_i=\varnothing$; 
\end{enumerate}

We now show that $[[\G]]_t$ is doubly infinitesimally generated. Let $g,g'\in [[\G]]_t$ be torsion elements (by Remark \ref{R: generating set infinitesimally generated}, it is enough to check that the condition in Definition \ref{D: infinitesimally generated} is satisfied for $g$ and $g'$ in the generating set consisting of torsion elements).
Let $U\subset X$ be a non-empty clopen set and consider a partition $U=U_1\cup U_1'$ into two non-empty clopen sets. 

Consider decompositions $g=g_1\cdots g_m$, sets $V_1,\ldots V_m$, elements $k_1,\ldots, k_m$ as above such that $k_i(V_i)\subset U_1$, and similarly $g'=g'_1\cdots g'_n$, $V_1', \ldots V_n'$, $k_1',\ldots, k'_n$ such that $k_i'(V_i')\subset U_1'$. Fix $i$ and $j$. Set $W_i=V_i\setminus U$ and $W'_j=V'_j\setminus (U\cup W_i)$. Then by construction, the four clopen sets $W_i, k_i(W_i), W'_j, k'_j(W_j)$ are all disjoint. Let $h\in [[\G]]_t$ be the element that coincides with $k_i$ on $W_i\cup k_i(W_i)$, with $k'_j$ on $W'_j\cup k'_j(W'_j)$ and with the identity elsewhere. Then $h\in [[\G]]_t$ and we have $h(V_i\cup V'_j)\subset U$, thereby concluding the proof.

\end{proof}

\begin{lemma}
The group $[[\G]]_t'$ is doubly infinitesimally generated and perfect.
\end{lemma}
The proof is a modification of an argument used by Cornulier (cf. the end of the proof of Th\'eor\`eme 3.1.6 in  \cite{Cornulier:Bourbaki}).

\begin{proof}
Let us say that an element $g\in[[\G]]_t$ is an $n$-\emph{cycle} if it has the following form.
$g$ has order $n$, and its support decomposes as a disjoint union of clopen sets
\[\supp(g)= V_1\sqcup \cdots \sqcup V_n,\]
such that $g(V_i)=V_{i+1}$ taking $i$ modulo $n$. 
Let $N$ be the subgroup of $[[\G]]_t$ generated by $3$ cycles (the same reasoning will apply for any $n\geq 3$ odd).
Then $N$ is normal in $[[\G]]_t$, since any conjugate of an $n$-cycle is still an $n$-cycle. It is easy to check that $N\neq \{e\}$. Namely  consider 3 points $x,y,z \in X$ lying in the same $\G$-orbit, and let $\gamma_1, \gamma_2\in \G$ be such that $s(\gamma_1)=x, r(\gamma_1)= y=s(\gamma_2), r(\gamma_2)=y$. Consider bisections $\mathcal{U}_1, \mathcal{U}_2$ containing $\gamma_1$ and $\gamma_2$. Let $V_1$ be a clopen neighborhood of $x$ small enough so that $V_1\subset s(\mathcal{U}_1)$, $V_1\cap\tau_{\mathcal{U}_1}(V_1)=\varnothing$,  $\tau_{\mathcal{U}_1}(V_1)\subset s(\mathcal{U}_2)$, and $\tau_{\mathcal{U}_2}\circ \tau_{\mathcal{U}_1}(V_1)\cap \tau_{\mathcal{U}_1}(V_1)=\varnothing$. The element $g\in [[\G]]_t$ that coincides with $\tau_{\mathcal{U}_1}$ on $V_1$, with $\tau_{\mathcal{U}_2}$ on $\tau_{\mathcal{U}_1}(V_1)$ and with $\tau_{\mathcal{U}_1}^{-1}\circ \tau_{\mathcal{U}_2}^{-1}$ on $\tau_{\mathcal{U}_2}\circ \tau_{\mathcal{U}_1}(V_1)$ and with the identity elsewhere is a non-trivial 3-cycle.

By Lemma \ref{L: infinitesimally generated} and Proposition \ref{P: normal contains commutator} it follows that $[[\G]]_t\subset N$.    

Moreover it is easy to see that every 3-cycle is the commutator of two 2-cycles. Hence $N\subset [[\G]]_t'$. Thus $[[\G]]_t'=N$.

Using the fact that $[[\G]]_t'$ is generated by 3 cycles, the same proof as in Lemma \ref{L: infinitesimally generated} can be repeated here to show that it is doubly infinitesimally generated (with a  minor modification to choose the elements $k_i$ there to be 3-cycles instead of involutions, thereby ensuring that $k_i\in [[\G]]_t'$). 

It remains to be seen that $[[\G]]_t'=[[\G]]_t''$. By Proposition \ref{P: normal contains commutator} and Lemma \ref{L: infinitesimally generated} it is enough to show that $[[\G]]_t''$ is non-trivial (since it is normal in $[[\G]]_t$). The same reasoning as above can be used to see that there exist non-trivial 5-cycles and that every  5-cycle is a commutator of 3-cycles, thereby belongs to $[[\G]]_t''$.

\end{proof}

\begin{remark}
The reason why we considered the group $[[\G]]_t'$ instead of $[[\G]]'$ is that we are not able to show, in general, that $[[\G]]$ is doubly infinitesimally generated. If one is able to show this, the same proof applies to show simplicity of the group $[[\G]]'$.

\end{remark}

\section{Preliminaries on groups acting on rooted trees}
\subsection{Basic definitions and bounded automata}
Let $T_d$ be the infinite $d$-regular rooted tree. The group of automorphisms of $T_d$ is denoted $\operatorname{Aut}(T_d)$.

 We fix an identification of vertices of $T_d$ with the set of words on the finite alphabet $\mathcal{A}=\{0,\ldots d-1\}$. The root of $T_d$ corresponds to the empty word $\varnothing$. We identify the symmetric group $S_d$ with the subgroup of $\aut(T_d)$ permuting the first level and acting trivially below.
 
 Every $g\in \aut(T_d)$ fixes the root and preserves the levels of the tree. Thus for every vertex $v\in T_d$ the automorphism $g$ induces an isomorphism between the sub-trees rooted at $v$ and at $g(v)$. The choice of an indexing of $T_d$ by the alphabet $\A$ allows to identify this isomorphism with a new element of $\aut(T_d)$, which is called the \emph{section} of $g$ at $v$ and is denoted $g|_v$.

A subgroup $G<\aut(T_d)$ is called \emph{self-similar} if for every $g\in G$ and every $v\in T_d$ we have $g|_v\in G$.
Any self-similar group admits a \emph{wreath recursion}, i.e. an injective homomorphism
\begin{align*}\label{E: wreath recursion}
G\hookrightarrow &G\wr_E S_d:=\bigoplus_E G \rtimes S_d\\
g\mapsto  &(g|_0,\ldots g|_{d-1})\sigma
\end{align*}

where $g|_0\cdots g|_{d-1}$ are the sections of $g$ at vertices at the first level (identified with the alphabet $E$) and the permutation $\sigma\in S_d$ gives the action of $g$ on the first level. 

If $G$ is a self-similar group, a the section at a vertex $v\in T_d$ defines a homomorphism
\[\varphi_v: \stab(v)\to G,\quad g\mapsto g|_v.\]
 
An important special case of self-similar groups are \emph{automaton groups}. A \emph{finite-state automaton} over the alphabet $\mathcal{A}$ is a finite subset $S\subset \aut(T_d)$ which is closed under taking sections: for every $g\in S$ and every $v\in \aut(T_d)$ we have $g|_v\in S$. Such a set can naturally be given the structure of an automaton in the more usual sense, see \cite{Nekrashevych:book}.

 The \emph{activity function} of an automaton $S$ is the function $p_S:\N\to \N$ that counts the number of vertices at level $n$ for which at least an element of $S$ has non-trivial section:
 \[p_S(n)=|\{v\in \mathcal{A}^n\: :\: \exists g\in S, g|_v\neq e\}|.\]
It can be shown that this function grows either polynomially with some well defined integer exponent $d\geq 0$, or exponentially. In the former case the integer $d$ is called the \emph{activity degree} of the automaton $S$. We will mostly be interested in the case $d=0$; in this case the function $p_S(n)$ is uniformly bounded in $n$ and the automaton is said to be of \emph{bounded activity} (for short, a \emph{bounded automaton}).

An \emph{automaton group} is a self-similar group $G<\aut(T_d)$ generated by a a finite-state automaton.

\subsection{The Schreier dynamical system of a group acting on a rooted tree}

Every level-transitive self-similar group has an associated \emph{Schreier dynamical system}, a well-studied object, see \cite{SchreierBasilica, Vorobets:Schreier}. 
It fits in the framework of \emph{uniformly recurrent subgroups} (URS), the topological analogue of an invariant random subgroup (IRS), recently introduced and studied by Glasner and Weiss \cite{Glasner-Weiss:URS}.

%
 
 Let $G$ be a countable group, and consider the space $\sub(G)$ of subgroups of $G$ endowed with the Chabauty topology (in the countable case this is simply the topology induced by the product topology on $\{0,1\}^G$). If $G$ is finitely generated, the choice of a finite symmetric generating set $S$ allows to identify $\sub(G)$ with a space of \emph{pointed labelled Schreier graphs} (where edges are labelled by generators in $S$), and the Chabauty topology coincides with the topology inherited by the space of pointed labelled graphs. The group $G$ acts on $\sub(G)$ by conjugation and this defines an action by homeomorphisms. If $\sub(G)$ is identified with the space of Schreier graphs with respect to a generating set $S$, the conjugation action corresponds to the action by ``moving the basepoint" as follows. Given a pointed labelled Schreier graph $(\Gamma, \gamma)$ and $g\in G$, choose a writing $g=s_n\cdots s_1$, with $s_i\in S$. Then $g\cdot (\Gamma, \gamma)=(\Gamma, g\gamma)$, where by definition $g\gamma$ is the endpoint of the unique path starting from $\gamma$ whose edges are labelled $s_1,\ldots, s_n$. 
 
 A \emph{uniformly recurrent subgroup}~\cite{Glasner-Weiss:URS}, or URS, is a nonempty closed minimal  $G$-invariant subset $X\subset \sub(G)$. 
 
 A construction of uniformly recurrent subgroups is given by the \emph{stability system} associated to a minimal $G$-dynamical system~\cite{Glasner-Weiss:URS}. Namely let $Y$ be a compact space and $G\acts Y$ be a minimal action by homeomorphisms. 
Consider the {stabiliser map}
\begin{align*}
Y\to \sub(G)\\
y\mapsto \stab(y).
\end{align*}
This map need not be continuous, however there is always a $G_\delta$-subset $Y_0\subset Y$ on which it is continuous, see  \cite[\S 1]{Glasner-Weiss:URS}.
The following proposition is proven in \cite{Glasner-Weiss:URS}.

\begin{prop}[Glasner-Weiss \cite{Glasner-Weiss:URS}]\label{P: stability system}
Let $G\acts Y$ be a minimal action of a countable group on a compact space by homeomorphisms, and $Y_0\subset Y$ the continuity locus of the stabiliser map, Then 
\[X=\overline{\{\stab(y)\: :\: y\in Y_0\}}\subset \sub(G).\]
is a URS, called the \emph{stability system} associated to the system $G\acts Y$.
\end{prop}

\begin{defin}[Grigorchuk]
Let $G$ be a group acting level-transitively on a rooted tree $T_d$. The \emph{Schreier dynamical system} of $G$ is th uniformly recurrent subgroup $X\subset \sub(G)$ given by the stability system for the action on the boundary of the tree.
\end{defin}
\begin{remark}
The assumption that the group acts level-transitively is equivalent to minimality of the action on the boundary of the tree, and thus Proposition \ref{P: stability system} applies. In particular, $G\acts X$ is also minimal.
\end{remark}
Let $G$ be a countable group acting by homeomorphisms on a compact space $Y$.
We say that a point $y\in Y$ is \emph{topologically regular} if every $g\in G$ that fixes $y$ fixes a neighbourhood of $y$ pointwise (in other words, if the isotropy group $\G_y$ for the groupoid of germs of the action is trivial). If every point is topologically regular, we say that the action $G\acts Y$ is topologically regular. 
\begin{lemma}\label{L: regularity implies continuity}
A point $y\in Y$ is topologically regular if and only if the stabiliser map is continuous at $y$.
\end{lemma}
\begin{proof}
 Topological regularity is equivalent to the fact that for every  $g\in G$ the map $Y\to \{0,1\}$ given by $z\mapsto 1_{gz=z}$ is constant on a neighbourhood of $y$ and hence continuous at $y$. It follows that the product map $Y\to \{0,1\}^G$, $z\mapsto (1_{gz=z})_g$ is continuous at $y$. But this is exactly the stabiliser map, after identifying $\sub(G)$ with a subset of $\{0,1\}^G$.  
\end{proof}

For groups generated by bounded automata, the Schreier dynamical system admits a fairly more explicit description. To give it we first recall some terminology.

Let now be a group $G$ acting on a rooted tree $T_d$. We say that a ray $\gamma=x_1x_2\cdots\in T_d$ is \emph{regular} if for every  $g\in G$ the section $g|_{x_1\cdots x_n}$ is eventually trivial, and \emph{singular} otherwise. The following Lemma is well-known and straightforward.
\begin{lemma}
 If a ray $\gamma\in \partial T_d$ is regular, then it is topologically regular, while the converse does not hold.
 \end{lemma}
We say that an orbit for the action of $G$ on $\partial T_d$ is regular if one (equivalently, all) of its points is regular, and \emph{singular} otherwise. The following lemma is also straightforward to check from the definition of activity.
\begin{lemma}
Let $G$ be a group generated by a bounded activity automaton. Then $G$ has finitely many singular orbits.
\end{lemma}
In particular, it follows from minimality that
\begin{cor}\label{C: regular ray}Let $G$ be a level-transitive bounded automaton group and let $\gamma\in \partial T_d$ be any regular ray. Then the Schreier dynamical system $X$ of $G$ is given by
\[ X= \overline{ \{\stab(g\gamma) \: | \: g\in G \}}.\]
\end{cor}

We now state the following definition.
\begin{defin}
Let $G$ be a bounded automaton group. The \emph{full group} of $G$, denoted $[[G]]$, is the topological full group of the Schreier dynamical system $G\acts X$.
\end{defin}

Recall that we denote $[[G]]_t$ the subgroup generated by torsion elements, and $[[G]]_t'$ its commutator subgroup. It immediately follows from Theorem \ref{T: simplicity} that
\begin{cor}\label{C: simple}
If $G$ is level-transitive, the group $[[G]]_t'$ is simple.
\end{cor}

The group $[[G]]_t'$ is not, in general, finitely generated, as the following lemma shows:

\begin{lemma}\label{L: basilica}
Assume that $G$ acts topologically regularly on   $\partial T_d$. Then the group $[[G]]_t'$ is not finitely generated unless $X$ is finite.
\end{lemma}
An example of a bounded automaton group acting topologically regularly on $\partial T_d$ is given by the Basilica group. In fact, it follows from results in \cite{SchreierBasilica} that for the Basilica group, the map $\stab: \partial T_d \to \sub(G)$ is a homeomorphism onto its image.
\begin{proof}
Recall that the action $G\acts X$ is a \emph{subshift} if there exists a finite partition of $X$ in clopen subsets such that the $G$-translates of this partition separate points. A necessary condition for the groups $[[G]], [[G]]'$ to be finitely generated is that the action $G\acts X$ is a subshift~\cite{Matui:simple, Cornulier:Bourbaki}, and the same argument applies to $[[G]]_t, [[G]]_t'$ as well.
Hence it is enough to show that if $G$ is as in the statement, then the Schreier dynamical system $G\acts X$ is not a subshift. In fact, it is a \emph{generalized odometer}: the orbit of every clopen set is finite. Indeed by Lemma \ref{L: regularity implies continuity}, topological regularity of $G\acts \partial T_d$ implies that  $\stab: \partial T_d\to X$ is a continuous equivariant surjective map. Let $U\subset X$ be clopen. It is enough to show that the orbit of $\stab^{-1}(U)$ is finite. But $\stab^{-1}(U)\subset \partial T_d$ is clopen, hence consists of a finite union of cylinder corresponding to a deep enough level of the tree. Since the $G$-action preserves levels of the tree, the conclusion follows. \qedhere

\end{proof}

The group $[[G]]_t'$ is however finitely generated in some cases. In the next section, we will study the group $[[G]]$ when $G$ is the \emph{alternating mother group}. In this case the group $[[G]]_t'$ is finitely generated and, as we shall see, this is enough to prove Theorem \ref{T: main}.

\section{The Schreier dynamical system of the mother group}
\label{S: amenability}
 Mother groups are a family of bounded automaton groups, first defined in~\cite{Bartholdi-Kaimanovich-Nekrashevych}, that contain all bounded automaton groups as subgroups \cite{Bartholdi-Kaimanovich-Nekrashevych, Amir-Angel-Virag:linear}. We work with a variant of the original definition defined only in term of alternating permutations, considered by Brieussel in \cite{Brieussel:folner}. 
\begin{defin}
Fix $d\geq 5$. The (alternating) \emph{mother group}  over $d$ elements is the group $M<\aut(T_d)$  generated by the two finite subgroups $A,B<\aut(T_d)$, where
\begin{itemize}
\item $A$ is the alternating group over $d$ elements acting on the first level with trivial sections.
\item $B$ is the set of all elements having a wreath recursion of the form
 \[g=(g, \sigma_1,\cdots, \sigma_{d-1})\rho,\]
 where $\sigma_1\cdots \sigma_{d-1}, \rho$ are alternating permutations  and $\rho$ fixes $0$.

\end{itemize}

\end{defin}

Observe that $B$ is a finite subgroup of $\aut(T_d)$, isomorphic to the permutational wreath product $A_d\wr_{\{1, \cdots, d-1\}} A_{d-1}$ where $A_d$ is the alternating group.

The interest of the mother group relies on the following fact:
\begin{thm}[\cite{Bartholdi-Kaimanovich-Nekrashevych}, \cite{Amir-Angel-Virag:linear}]\label{T: mothergroup contains all}
Let $G$ be a bounded automaton group. Then $G$ embeds in an alternating mother group (over a possibly bigger alphabet).
\end{thm}
\begin{remark}
In fact, the mother groups defined in \cite{Bartholdi-Kaimanovich-Nekrashevych, Amir-Angel-Virag:linear} are slightly different: generators are defined by  same recursive rules but without the constraint that the permutations involved belong to the alternating group. However, it is straightforward to see that the mother group over $d$ elements embeds in the alternating mother group over $2d$ elements.
\end{remark}

 
The following fact is proven by Brieussel in~\cite{Brieussel:folner}.
 
\begin{lemma}[Proposition 3.1 in~\cite{Brieussel:folner}]\label{L: Brieussel}
 If $d\geq 5$ the wreath recursion map defines an isomorphism 
 \[M\tilde{\to}  \oplus_\mathcal{A} M\rtimes A_d= M\wr_\mathcal{A}A_d\]
 where $A_d$ is the alternating group.
 \end{lemma}

A consequence (that can also be proven directly by induction) is the following
\begin{cor}\label{C: totally non free}
The action $M\acts \partial T_d$ is totally non-free (two distinct points have different stabilizers).
\end{cor}
\begin{proof}
Let $\gamma\neq \eta\in \partial T_d$ and let $w$ be their biggest common prefix. Let $x,y\in \mathcal{A}$ be the letters following $w$ in $\gamma, \eta$. Using Lemma \ref{L: Brieussel}, one can find $g\in M$ so that $g|w$ is an element of $A$ given by a permutation $\sigma$ such that $\sigma(x)=x$ and $\sigma(y)\neq y$.  Hence $g\gamma=\gamma$ but $g\eta\neq \eta$. \qedhere
\end{proof}

We also collect here same well-known facts about the group $M$ that we will use. They can be easily proven by induction, or proofs can be found e.g. in \cite{Amir-Virag:positivespeed}.
\begin{lemma}\label{L: properties mother group}
\begin{enumerate}
\item $M$ acts level-transitively on $T_d$.
\item Two rays $\gamma, \eta\in \partial T_d$ are in the same $M$-orbit if and only if they are cofinal. Moreover there is only one singular orbit, namely the orbit of $\gamma=0^\infty$.
\item Let $\rho=x_0x_1\cdots \in \partial T_d$ be in the orbit of $0^\infty$  and $g\in M$. Then the section $g|_{x_0\cdots x_n}$ is eventually constant and belongs to $B$. We denote $g|_\rho$ its eventual value.
\item The group $M$ is \emph{contracting} with nucleus $A\cup B$: for every $g\in M$ there exists $n$ so that all sections of $g$ at levels $r\geq n$ belong to  $A\cup B$.
\end{enumerate}
\end{lemma}

From now on, we shall fix $d\geq 5$, and denote $M\acts X$ the Schreier dynamical system of $M$, and $\M$ its groupoid of germs. We further denote $[[M]]$ the topological full group of $\M$. Recall that we denote $[[M]]_t$ the subgroup of $[[M]]$ generated by torsion elements, and $[[M]]_t'$ the commutator subgroup of $[[M]]_t$.
We have
\begin{lemma}
 If $d\geq 6$, $M$ embeds in $[[M]]_t'$.
 \end{lemma}
 \begin{proof}
First observe that the action $M\acts X$ is faithful. Namely let $\gamma \in T_d$ be a regular ray and $O(\gamma)$ be its orbit. By Lemma \ref{L: properties mother group} 1 $O(\gamma)$ is dense in $\partial T_d$ and thus $M\acts O(\gamma)$ is faithful. By Lemma \ref{C: regular ray} the space $X$ is given by the closure of stabilizers of points in $O(\gamma)$, and by Corollary  \ref{C: totally non free}Ê the stabiliser map restricted to $O(\gamma)$ is injective. It follows that $X$ contains an invariant subset on which the action is faithful. Hence $M$ embeds in $[[M]]$. To check that it is actually contained in $[[M]]_t'$, it is enough to show that generators in $A$ and $B$ can be written as commutators of torsion elements. This is obvious for generators in $A$ (since $A$ is the alternating group over $d\geq 6$ elements), and observe that $B$ is generated by its subgroups $B_0, B_1,\ldots B_{d-1}$ where $B_0$ is the subgroup consisting of elements with all the $\sigma_i$ trivial (it is thus isomorphic to an alternating group over $d-1\geq 5$ elements) and for $1\leq i \leq d-1$ $B_i$ is the subgroup of $B$ consisting of elements that have $\rho$ and $\sigma_j$ trivial for all $j\neq i$ (it is thus isomorphic to the alternating group over $d$ elements).

\end{proof}

 Thus, Theorem \ref{T: main} follows from the combination of Theorem \ref{T: mothergroup contains all} with
\begin{thm}\label{T: full group mother group}
The group $[[M]]_t'$ is a finitely generated, simple amenable group.
\end{thm} 
The rest of the paper is devoted to the proof of this result. Simplicity has already been established (see Corollary \ref{C: simple}).

\subsection{Bratteli diagram representation and amenability}

Let $B$ be a Bratteli diagram. Given a vertex $v$ of $B$, recall that we denote $T_v$ the \emph{tower} corresponding to $v$. Recall that this is the collection of all cylinders subsets of $X_B$ corresponding to paths ending in $v$.

\begin{defin}
 A homeomorphism $g$ of $X_B$ is said to be \emph{of bounded type} if for every $v\in B$ the number of cylinders $C_\gamma\in T_v$ so that $g|_{C_\gamma}$ is not equal to $\tau_{\gamma, \gamma'}$ for some $\gamma'$ is bounded uniformly in $v$, and the set of points $x\in X_B$ such that the germ of $g$ in $x$ does not belong to $\H_B$ is finite. 
\end{defin}
The following result is due to Juschenko, Nekrashevych and de la Salle \cite[Theorem 4.2]{Juschenko-Nekrashevych-Salle:recurrentgroupoids}. 

\begin{thm}[\cite{Juschenko-Nekrashevych-Salle:recurrentgroupoids}] \label{T: JNS}
Let $G$ be a group of homeomorphisms of bounded type of the path space $X_B$ of a Brattelli diagram, and $\mathcal G$ be the groupoid of germs of $G\acts X_B$. Assume that for every $x\in X_B$ the isotropy group $\G_x$ is amenable. Then $G$ is amenable. Moreover $[[\G]]$ is amenable.

\end{thm}
Recall that $M\acts X$ is the Schreier dynamical system of the mothergroup. We denote $\M$ its groupoid of germs. We have:
\begin{thm}\label{T: Bratteli}
There exists a stationary, simple Bratteli diagram $B$ and a homeomorphism $X\simeq X_B$ that conjugates the action $M\acts X$ to an action by homeomorphisms of bounded type. Moreover
 the action $M\acts X_B$ is topologically regular (equivalently, $\M_x$ is trivial for every $x\in X_B$).
\end{thm}

By Theorem \ref{T: JNS}, this implies:
\begin{cor}
The group $[[M]]$ is amenable.
\end{cor}

Before proving the theorem, let uus first describe the idea of the construction of the Bratteli diagram in Theorem \ref{T: Bratteli} and fix some notation. The path space of the diagram $B$ will be obtained from the boundary of the tree by ``removing'' the orbit of the ray $0^\infty$ and replacing it by finitely many ones with two extra letters ``at infinity''.

 More precisely, the path space of $B$ will be in bijection with the set $\widetilde{X}$ defined as follows. let $O\subset \partial T_d$ be the orbit of the zero ray $0^\infty$. Recall from Lemma \ref{L: properties mother group} that $O$ consists exactly of rays that are co-final with $0^\infty$, and that $O$ is the only singular orbit of $M$. For $a, b\in \A=\{0, \ldots , d-1\}$ with $a\neq 0$, denote $O_{ab}$ the set of formal words of the form $\rho ab$ where $\rho \in O$. We say that the two letters $ab$ lie ``at infinity'' in positions $\omega, \omega+1$. Set $O_*=\cup O_{ab}$, where the union is taken over all $a, b\in \mathcal{A}$ with $a\neq 0$. 

\begin{defin}
With the notations above, we denote \label{D: tilde X} the set
\[\widetilde{X}= (\partial T_d \setminus O)\cup O_*.\]
We make the group $M$ act on  $\widetilde{X}$ as follows. The action on $\partial T_d \setminus O$ is given by the restriction of the action on the boundary of the tree. If $\rho ab \in O_*$ and $g\in M$ we set $g(\rho ab)=g(\rho)g|_{\rho}(ab)$, where the section $g|_\rho$ is as in point 3 of Lemma \ref{L: properties mother group}.
\end{defin}
Note that we do not consider any topology on $\widetilde{X}$ yet. We will now construct a Bratteli diagram $B$ so that the path space $X_B$ is in bijection with $\widetilde{X}$, and then we will consider on $\widetilde{X}$ the topology induced by this bijection. We will then show that $M\acts \widetilde{X}=X_B$ (with this topology) is conjugate to the Schreier dynamical system $M\acts X$.

\subsubsection*{Construction of the diagram}

 We define a Bratteli diagram $B$ as follows. The level set $V_0$ consists of a single vertex, the \emph{top vertex}. All levels $V_n, n\geq 1$ are identified with a copy of the set $\{(ab,i)\}$ where $a\in \mathcal{A}\setminus\{0\}, b\in \mathcal{A}$ and $i\in \{0, \ast\}$.
Every vertex $(ab, 0)$ at level $n$ is connected with the vertices $(ab, 0)$ and $(ab, \ast)$ at level $n+1$. We label  edges of this form by the symbol $0$. Every vertex of the form $(ab, \ast)$ with $b\neq 0$ is connected to all vertices of the form $(bc, \ast)$ with $c$ arbitrary. We label edges of this type by with the symbol $a$.
Finally every vertex of the form $(a0, \ast)$ is connected to all vertices of the form $(cd, 0)$ with $c\neq 0$ and $d$ arbitrary. These edges are also labelled $a$. The top vertex is connected to each vertex of $V_1$ by $d$ edges, labelled by $\{0, \ldots, d-1\}$. Let $B$ be the Bratteli diagram obtained in this way. Note that $B$ is simple.

The path space $X_B$ is in bijection with $\widetilde {X}$ as follows. Every path $\gamma\in X_B$ corresponds to the sequence $\tilde{\gamma}\in \widetilde{X}$ read on the labels of its edges, if this sequence does not end with infinitely many zeros. If the sequence read ends with infinitely many $0$s, then observe that the path $\gamma$ must end with an infinite sequence of edges between vertices of the form $(ab, 0)$ for some fixed $a\neq 0, b\in \mathcal{A}$, and in this case we put the two letters $ab$ ``at infinity''. Conversely, given any sequence $\tilde{\gamma}\in \widetilde{X}$ one can find a unique path $\gamma\in X_B$, so that the labels of its first $n$ edges coincide with the first $n$ letters of $\tilde{\gamma}$, and the $n$-th vertex $v_n=(ab, i)\in V_n$ of $\gamma$ encodes the following information: the symbol $i\in \{0,\ast\}$ tells wether the $n+1$th letter is zero or non-zero,  $a$ is the first non-zero letter in $\gamma$ appearing after position $n$ and $b$ is the letter following $a$.

The group $M$ acts on on $X_B$ through the identification $\widetilde{X}\simeq X_B$. We will systematically use this identification.
We put on  $\widetilde{X}$ the topology induced by this identification. Let us describe this topology in a more explicit way. 
\begin{lemma}\label{L: topology tilde X}
Let $v\in B$ be a vertex at level $n$ and $\eta$ be a finite path ending at $v$. Let $w\in \mathcal{A}^n$ be the sequence of labels read on $\eta$. Consider the cylinder subset $C_\eta\subset X_B$, and view it as a subset of $\widetilde{X}$ through the bijection $X_B\simeq \widetilde{X}$ described above.
Then 
\begin{enumerate}
\item If $v$ has the form $(ab, \ast)$ then $C_\eta$ consists exactly of all sequences starting with $wab$.
\item If $v$ has the form $(ab, 0)$ then $C_\eta$ consists exactly of all sequences in $\widetilde{X}$ starting with a prefix of the form $w0^n$ for some $n\in \N_+\cup\{\infty\}$, and such that the first non-zero letter following this prefix is $a$ and the next letter is $b$ (the letters $ab$ could be ``at infinity'').
\end{enumerate}

\end{lemma}
\begin{proof} Follows directly by inspecting  the constructed  bijection $\widetilde{X}\simeq X_B$.\qedhere \end{proof}

\begin{lemma}
The action of $M$ on $X_B$ is by homeomorphisms of bounded type.
\end{lemma}
\begin{proof}
Pick a vertex $v$ in $B$ and let $\eta$ be a path ending at $v$. Let $w$ be the word read on $\eta$, and note that $\eta$ is the unique path ending at $v$ on which $w$ is read.  Let $g\in M$. It follows from the definition of the action $M\acts \widetilde{X}$ and from Lemma \ref{L: topology tilde X} that $g$ restricted to $C_\eta$ coincides with a transformation of the form $\tau_{\eta, \eta'}$ unless $g|_w\neq e$. Since $M$ has bounded activity, there are a bounded number of such $w$.

Moreover  for elements in the standard generating set of $M$, the germs at every point of $X_B$ belongs to $\mathcal{H}_B$ except perhaps for points that correspond to sequences of the form $0^\infty ab\in \widetilde{X}$, that may be sent to sequences of the same form for a different  choice of $ab$. The conclusion follows.
\end{proof}
\begin{lemma}\label{L: X_B topologically regular}
The action $M\acts X_B$ is topologically regular. 
\end{lemma}
\begin{proof}
Let $\gamma\in X_B$ and $g\in M$ so that $g\gamma=\gamma$. View $\gamma$ as an element of $\widetilde{X}$. By Lemma \ref{L: properties mother group} 3, there exists an $n$ so the section of $g$ at $n$th prefix of $\gamma$ belongs to the finite subgroup $B$. Consider the cylinder $C_{\gamma_n}\subset X_B$ corresponding to the first $n$ edges of the path $\gamma$ (now viewed as a path in the Bratteli diagram). We claim that $g$ fixes $C_{\gamma_n}$ point-wise. The reason is that $C_{\gamma_n}$ is reminiscent of the letters $ab$, where $a$ is the first letter that follows the $n$th position in $\gamma$ which is non-zero, and $b$ is the letter following $a$ ($b$ may be 0). An element of $B$ only acts on the first non-zero letter of a ray and on the next one. Since $g$ fixes $\gamma$ and its section at the $n$th level belong to $B$, it follows that $g$ fixes $C_{\gamma_n}$ pointwise.\qedhere
\end{proof}
\begin{proof}[Proof of Theorem \ref{T: Bratteli}]
Consider the stabiliser map $\stab:X_B\to \sub(G)$. This map is continuous by Lemma \ref{L: X_B topologically regular} and Lemma \ref{L: regularity implies continuity}. Moreover its image is exactly the Schreier dynamical system $X$, since $X_B$ contains an invariant dense subset on which the action of $M$ is conjugate to the action on $\partial T_d\setminus O$. Hence we only need to check that it is injective. Let $\gamma\neq\gamma'\in X_B$. It is easy to construct, $g\in M$ so that $g\gamma=\gamma$ and $g\gamma'\neq \gamma'$ (e.g. using Lemma \ref{L: Brieussel}) and thus $\stab(\gamma)\neq \stab(\gamma')$.  \qedhere

\end{proof}
 \begin{remark}\label{R: trivial normalizer}
 It follows from the proof that every Schreier graph in $X$ has no non-trivial automorphism (as an unrooted labelled graph). Indeed the proof shows that $X$ coincides with its stability system, and thus every element of $X$ (regarded now as a subgroup of $M$) coincides with its normalizer in $M$.
 
 \end{remark}
\section{Finite generation}

\label{S: finite generation}
In this section we show
\begin{thm} \label{T: finitely generated}
$[[M]]_t'$ is finitely generated.
\end{thm}
We provide two proofs:  the first  (that was added after the appearance of \cite{Nekrashevych:groupoidsimple}) consists in using the explicit description of $M\acts X$ obtained in the previous section to show that the action $M\acts X$ is conjugate to  subshift, and then applying a theorem of Nekrashevych from \cite{Nekrashevych:groupoidsimple}. The second  consists in constructing an explicit  generating set and is based on combinatorial properties of Schreier graphs of $M$.

We denote points in the Schreier dynamical system $X$ with the notation $(\Gamma, \rho)$, where $\Gamma$ is a labelled Schreier graph and $\rho\in \Gamma$ is its basepoint. 

Recall that a group action by homeomorphisms on the Cantor set $G\acts X$ is said to be conjugate to a \emph{subshift} over a finite alphabet, if there exists a finite partition of $\mathcal P$ of $X$ into clopen sets such that the $G$-translates of elements of $\mathcal{P}$ separate points.

 \begin{lemma}\label{L: subshift}
 $M\acts X$ is  conjugate to subshift over a finite alphabet.
 \end{lemma}
 \begin{proof}
We define a clopen partition $\mathcal{P}$ of $X$ so that $M$-translates of $\mathcal{P}$ separate points. By definition $(\Gamma, \gamma)$ and $(\Gamma', \rho)$ are in the same element of $\mathcal{P}$ if the loops at $\gamma, \rho$ have the same labels. Let us show that the $M$-translates of $\mathcal{P}$ separate points. Use the identification $X\simeq \widetilde{X}$ introduced in the previous section (see Definition \ref{D: tilde X}). Let $\gamma\neq \rho \in \widetilde{X}$. Since $\gamma\neq \rho$, there is a first bit in which $\gamma$ and $\rho$ differ, say a position $r\in \N\cup \{\omega, \omega+1\}$.  Assume at first that $r\in \N$. Let this bit be $x$  in $\gamma$ and $y$ in $\rho$. If $r=1$ then there are generators in $A$ that fix $\gamma$  but not $\rho$ and we are done. Otherwise let $w$ be the common prefix of length $r+1$. Since the group $M$ acts level-transitively on $T_d$ and using Lemma \ref{L: Brieussel}, we can find $g\in M$ so that $g(w)=0^{r-2}1$ and $g|_w=e$. Then $g(\gamma)= 0^{r-2}1x\cdots$ and $g(\rho)=0^{r-2}1y\cdots$. It is easy to see that the generators in $B$ that fix $g(\gamma)$ and $g(\rho)$ are different and thus $g(\gamma), g(\rho)$ lye in different elements of $\mathcal{P}$. The case $r\in \{\omega, \omega+1\}$ is similar, but choose instead $g$ so that $g(w)=0\cdots 0$ (where $w$ is the non-zero prefix common to $\gamma$ and $\rho$).\qedhere

 \end{proof} 
 
 \begin{remark}
The mother group plays an important role in the previous proof: by the proof of Lemma \ref{L: basilica}, the Schreier dynamical system of a bounded automaton group is not always conjugate to a subshift.
 \end{remark}
 
 \begin{proof}[First proof of Theorem \ref{T: finitely generated}]
By Nekrashevych's  result \cite{Nekrashevych:groupoidsimple}, Lemma \ref{L: subshift}Ê  is enough to conclude the proof. Namely it is proved in \cite{Nekrashevych:groupoidsimple} that if $\G$ is an expansive groupoid with infinite orbits, the alternating subgroup $A(\G)<[[\G]]$ is finitely generated (see \cite{Nekrashevych:groupoidsimple} for the definition of $A(\G)$), and that the groupoid of germs of a subshift is expansive. In this situation the group $A(\G)$ coincides with $[[M]]_t'$, indeed it is a  normal subgroup of $[[M]]_t'$ and the latter is simple by Theorem \ref{T: simplicity}. This is enough to conclude.
 \end{proof}

We now turn to the second proof. 

\subsection{Basic properties of Schreier graphs of $M$}

In this subsection we collect some preliminary considerations on Schreier graphs of $M$.

We let $\widetilde{X}= (\partial T_d \setminus O)\cup O_*$ be the space of sequences introduced in Definition \ref{D: tilde X}. We will often use the identification $X\simeq \widetilde{X}$.

Fix $\rho\in \widetilde{X}$ and let $O(\rho)$ be its orbit. Denote $\Gamma$ the corresponding Schreier graph, with respect to the generating set $S=A\cup B$.
\subsubsection*{Projection to the Gray code line}

For $\gamma\in \widetilde{X}$, let $\overline{\gamma}$ be the sequence in the binary alphabet $\{0, \ast\}$ obtained by replacing all non-zero letters of $\gamma$ by $\ast$.

This defines a graph projection $p:\Gamma \to \oGamma$ where $\oGamma$ is the following graph. Its vertex set consists of sequences $\overline{\gamma}$ where $\gamma\in O(\rho)$. Two such sequences are neighbour if and only if they differ either in the first bit only, or in the bit that follows the first appearance of $\ast$ only. Edges of the first type are called edges `` of type $A$'', and edges of the second type are called edges `` of type $B$''. Since every vertex has exactly two neighbours, $\oGamma$ is a line.
The projection $p:\Gamma\to \oGamma$ preserves adjacency. More precisely, edges given by actions of generators in $A$ project to edges of type $A$ (unless its endpoints are mapped to the same vertex), and the same with $B$.

It is convenient to think of the graph $\Gamma$ as `` fibered over'' the line $\oGamma$.

\begin{defin}
We call $\oGamma$ the \emph{Gray code line} associated to $\Gamma$. 
\end{defin}

The reason for the name is the connection to the Gray code ordering of binary sequences. This connection, and the fact that Schreier graphs of $M$ projects to lines was used in \cite{Amir-Virag:positivespeed, Amir-Virag:speedexponent, Liouvilletrees}. 
\begin{defin}
Let $\gamma\in \widetilde{X}$. We say that the bit at position $r\in \N\cup\{\omega, \omega+1\}$  is \emph{visible} in $\gamma$ if it is either the first bit (in which case we say that it is $A$-visible), or if it is the first non-zero bit, or the following one (in which cases we say that it is $B$-visible).

The same terminology applies to bits in the projected sequence $\overline{\gamma}$.

If $\overline{I}\subset \overline{\Gamma}$ is a segment, we say that a bit is visible in $\overline{I}$ if it is visible in at least one sequence in $\overline{I}$.
\end{defin}
\begin{remark}\label{R: visible}
Acting on $\gamma$ with a generator in $A, B$ only modifies a bit that is $A, B$-visible (this is a straightforward consequence of the definition of the groups $A, B$).

\end{remark}

The following definition will allow us to define a basis for the topology of the Schreier dynamical system $X$, particularly convenient to work with.
\begin{defin}[Gray code piece] \begin{enumerate}
\item A \emph{Gray code piece} $I$  a subgraph of $\Gamma$ which is a connected component of $p^{-1}(\overline I)$, where $\overline{I}=[\overline{\gamma}_0,\cdots \overline{\gamma}_{n-1}]\subset\overline \Gamma$ is a finite segment. We still denote $p: I\to \overline{I}$ the restriction of the projection map to $I$. The \emph{length} of a Gray code piece is the length of $\overline{I}$ (i.e. the number of its vertices).

\item A \emph{pointed Gray code piece} $(I, \gamma)$ is a Gray code piece $I$ together we a preferred vertex $\gamma\in I$. We will always denote $\overline v\in\overline{I}$ the projection of $v$ to $\overline I$.

\item A pointed Gray code piece $(I, \gamma)$ is said to be \emph{central} if $\gamma$ projects to the midpoint of $\overline I$ (in particular, a central Gray code piece has odd length).
\end{enumerate}
\end{defin}

\begin{lemma} Gray code pieces are finite graphs.
\end{lemma}
\begin{proof}

Let $I\subset \Gamma$ be a Gray code piece, and pick a vertex $\gamma\in I$. Since $I$ is connected, every vertex $\gamma'$ of $I$ can be reached from $\gamma$ acting with generators in $A \cup B$ and without never getting out of $\overline{I}$ in the projection. Thus, the corresponding sequence $\gamma'$ only differs from $\gamma$ in bits that are visible in  $\overline{I}$. Thus there are only finitely many possibilities for $\gamma'$. \qedhere
\end{proof}

 \begin{defin}[Marginals]
Let $(I, \gamma)$ be a central Gray code piece of length $2n+1$ with $n\geq 2$ and $\overline{I}=[\overline{\gamma}_{-n},\cdots , \overline{\gamma}_0, \cdots \overline{\gamma}_n]$. Denote $\overline I_l=[\overline{\gamma}_{-n}, \cdots \overline{\gamma}_{n-2}]\subset\overline I$ 
the segment consisting of the $2n-1$ leftmost vertices of $\overline I$ and by $\overline I_r$ the segment consisting of the rightmost $2n-1$ vertices. Let $I_l, I_r\subset I$ be respectively the connected components of $\gamma$ in $p^{-1}(\overline I_l), p^{-1}(\overline I_r)$. Then $({I}_l, \gamma), ({I}_r, \gamma)$ are two (non-central) pointed Gray code pieces, that we call the \emph{marginals} of $(I, \gamma)$.

We also call $\overline{I_c}\subset \overline{I}$ the segment consisting of the $2n-3$ central vertices, and let $I_c\subset I$ be the connected component of $\gamma$ in $p^{-1}(I_c)$, so that $(I_c, \gamma)$ is a central Gray code piece of length $2n-3$.
\end{defin}

The main combinatorial feature of the graphs that we will use is contained in the next proposition.

\begin{prop}\label{P: marginals}
Let $(I, \gamma)\subset \Gamma$ be a centered Gray code piece. Then the isomorphism class of $(I, \gamma)$ as a labelled graph is uniquely determined by the isomorphism classes of its marginals.
\end{prop}

The proof of this proposition requires a more detailed analysis of the Schreier graphs, which will not be relevant for the rest of the proof. For this reason, it is postponed to Subsection~\ref{S: proof of prop}.
\subsection{Second proof of Theorem~\ref{T: finitely generated}}

 The strategy is to follow the same reduction steps as in Matui's proof  \cite[Theorem 5.4]{Matui:simple} of finite generation for the topological full group of a minimal $\Z$-subshift. However, the analysis is substantially more involved, as we need to deal with the more complicated nature of the graphs.

 We define a convenient basis for the topology on $X$.
\begin{defin}[Gray code cylinder]
\begin{enumerate}
\item Let $(I, \gamma)$ be a pointed Gray code piece. 
The corresponding \emph{Gray code cylinder} $C_{I, \gamma}\subset X$ the collection of all $(\Gamma, \rho)\in X$ so that a Gray code piece around $\rho$ is isomorphic to $(I, \gamma)$. We say that a Gray code cylinder is \emph{central} if the corresponding Gray code piece is central.
\item For $(\Gamma, \rho)\in X$ we denote $(\Gamma|_n, \rho)$ the pointed central Gray code piece of length $2n+1$ around $\rho$. 

\end{enumerate}
\end{defin} 
Central Gray code cylinders are a basis of clopen sets for the topology on $X$.

The following is a consequence of Proposition~\ref{P: marginals}.
\begin{lemma}\label{L: marginals and cylinders}
Let $(I, \gamma)$ be a central Gray code piece, and $(I_l, \gamma), (I_r, \gamma)$ be its marginals. Then
\[C_{I, \gamma}=C_{I_l, \gamma}\cap C_{I_r, \gamma}.\]
\end{lemma}

\begin{proof} The inclusion $\subset$ is obvious. The reversed inclusion is a consequence of Proposition \ref{P: marginals}.\qedhere \end{proof}

We define a partial action of $M$ on pointed Gray code pieces. The action of $g\in M$ on $(I, \gamma)$ is defined if and only if it is possible to write $g$ as  $g=s_n\cdots s_1$ with $s_i\in S$, in such a way that the path in $I$ starting from $\gamma$ with labels $s_1,\cdots, s_n$ is entirely contained in $I$. In that case we set $g(I, \gamma)=(I, g\gamma)$ where $g\gamma$ is the endpoint of the above path.  

\begin{defin}\begin{enumerate}
\item Given a clopen subset $U\subset X$, and elements $g,h\in M$ we call the triplet $(U,g,h)$ \emph{admissible} if the sets $U,g^{-1}U,hU$ are disjoint.
\item Given an admissible triplet $(U,g,h)$ we define $\eta_{U,g,h}$ to be the element of $[[M]]$ that acts as $g,h,g^{-1}h^{-1}$ on $g^{-1}U,U,hU$ respectively, and the identity elsewhere.
\end{enumerate}
\end{defin}

\begin{lemma}\label{L: admissible commutator} Let $(U, g, h)$ be an admissible triplet. Then $\eta_{U,g,h}\in
[[M]]_t'$.
\end{lemma}
\begin{proof}  Clearly $\eta_{U, g, h}$ belongs to a subgroup of $[[M]]$ isomorphic to the symmetric group $S_3$ (which is thus contained in $[[M]]_t$) and corresponds to a 3-cycle (which is thus a commutator). \end{proof}

\begin{lemma}
The elements $\eta_{U,g,h}$ generate $[[M]]_t'$ as $(U, g, h)$ varies among admissible triplets.
\end{lemma}

\begin{proof}
As in \cite{Matui:simple}, the starting observation is that since $[[M]]_t'$ is simple, it is generated by its elements of order 3 (note that there are such elements, for instance by Lemma \ref{L: admissible commutator}). Let $k\in [[M]]_t$ with order 3. Let us show that it can be written as a product of elements of the form $\eta_{U, g, h}$. Since the action of $M$ on $X$ is topologically regular (Theorem~\ref{T: Bratteli}), the set of fixed points of $k$ is clopen. Let $V$ be its complement. $V$ is covered by clopen sets of the form $W=W_1\sqcup W_2\sqcup W_3$ so that $k$ permutes cyclically the $W_i$ and so that the restriction of $k$ to each $W_i$ coincides with the restriction of an element of $M$. After taking a finite sub-cover and refining it, we may suppose that $V$ is partitioned into clopen sets $W$ of this form. Note that given such a set $W$, the element of $[[M]]$ that acts as $k$ on $W$ and as the identity elsewhere can be written in the form $\eta_{W_1, g, h}$ for some $g, h\in M$. Thus $k$ can be written as a product of commuting elements of this form.

\end{proof}
\begin{lemma}\label{L: disjoint cylinders}
For every $R>0$, there exists $n>0$ such that the following holds. Let $(\Gamma, \gamma) \in X$. Assume that $\rho\neq \gamma$ is a vertex of $\Gamma$ lying at distance less than $R$ from $\gamma$. Then $(\Gamma|_{n}, \gamma)\neq (\Gamma|_{n}, \rho)$. \end{lemma}
\begin{proof}
Assume that there is a sequence $(\Gamma_n, \gamma_n, \rho_n)$ verifying the assumptions so that for every $n$ we have $(\Gamma_n|_n, \gamma_n)=(\Gamma_n|_n, \rho_n)$. Up to taking a subsequence we may assume that $(\Gamma_n, \gamma_n)$ and $(\Gamma_n, \rho_n)$ converge to limits $(\Gamma, \gamma)$ and $(\Gamma, \rho)$, where the limit graph $\Gamma$ is the same with two different basepoints (here we use the assumptions that $\gamma_n$ and $\rho_n$ are at bounded distance in $\Gamma_n$). Then $\gamma$ and $\rho$ are indistinguishable in $\Gamma$. Arguing as in the proof of Lemma~\ref{L: first reduction}, this contradicts Remark~\ref{R: trivial normalizer}, since then $\Gamma$ has a non-trivial automorphism.  \qedhere
\end{proof}

We will need the following simple fact.

\begin{lemma}
Let $\Delta$ be a finite connected graph, with vertex set $\{1, \ldots, n\}$. Then the alternating group over $n$ elements is generated by 3-cycles of the form $(x,y,z)$, where $x, y, z\in \{1,\ldots, n\}$ are the vertices of a simple path of length 3 in $\Delta$. 
\end{lemma}

\begin{lemma}\label{L: first reduction}
The elements $\eta_{U,s,t}$ generate $[[M]]_t'$ as $(U,s,t)$ varies among admissible triplets so that $s,t\in S=A \cup B$.

\end{lemma}
\begin{proof}
Consider $\eta_{g,h, U}$ with $g,h$ arbitrary. By writing $U$ as a disjoint union of central Gray code cylinders, we may suppose $U=C_{I, \gamma}$  is a central Gray code cylinder whose Gray code piece $(I, \gamma)$ is such that the length of $I$ is bigger than $2\max\{|g|_S, |h|_S\}$ so that $(I, g\gamma)$ and $(I, h^{-1}\gamma)$ are defined, and $gU$ and $h^{-1}U$ are the corresponding Gray-code cylinders.
Let $\Delta\subset I$ be a minimal connected subgraph containing $h^{-1}\gamma, \gamma$ and $g\gamma$, and let $R$ be its diameter.  and let $n_0$ be given by Lemma~\ref{L: disjoint cylinders}. Up do decomposing again into Gray code cylinders of bigger length, we may assume that the length of $I$ is at least $n_0+R$. Then Lemma~\ref{L: disjoint cylinders} guarantees the following: if $\delta, \delta'\in \Delta$ are distinct, then the Gray code cylinders $C_{I, \delta}$ and $C_{I, \delta'}$ are disjoint. Hence $[[M]]$ contains a copy of the symmetric group acting over $|\Delta|$ elements that acts by permuting such cylinders. Then $\eta{g, h, U}$ corresponds to the 3-cycle $(h^{-1}\gamma, \gamma, g\gamma)$. Moreover 3-cycles permuting adjacent elements of $\Delta$ are of the form $\eta_{C_{I, \delta}, s, t}$ with $s,t\in S$. Hence $\eta_{g, h, U}$ belongs to the group generated by such elements. \qedhere

%

\end{proof}

From this point on, let $n_0$ be the integer furnished by Lemma~\ref{L: disjoint cylinders} for $R=8$.
In the next definition and the following lemma, it will be convenient slightly extend the generating set by setting $\widetilde{S}=S^2$. Note that this is still a generating set since $e\in S$.

\begin{defin}\label{D: convenient admissible}
We say that  $(U, s, t)$ is a \emph{convenient admissible triplet} if it is an admissible triplet that has the following form: $U=C_{I, \gamma}$ is a central Gray code cylinder where $I$ has length at least $2n_0+1$, $s, t\in \widetilde{S}$, and $s^{-1}\gamma, \gamma, t\gamma$  project to three distinct consecutive points in the Gray code line. 

\end{defin}

The following Lemma is based on the same commutator trick as in~\cite[Lemma 5.3]{Matui:simple}. The possibility to apply it in this situation relies on the combinatorial nature of Gray code cylinders (Proposition~\ref{P: marginals} and Lemma \ref{L: marginals and cylinders}).
\begin{lemma}\label{L: commutator trick}
Let $(U, s,t)$ be a convenient admissible triplet, with $U=C_{I, \gamma}$. Let $(I_l, \gamma),  (I_r, \gamma)$ be the marginals of $(I, \gamma)$. Observe that the definition of convenient admissible triplet implies that $(I_l, s^{-1}\gamma)$ and $(I_r, t\gamma)$ are central Gray code pieces of smaller length.  Set $U_l=C_{I_l, s^{-1}\gamma}$ and $U_r=C_{I_r, t\gamma}$. Choose $s', t'$ such that the triplets $(U_l, s', s)$ and $(U_r, t, t')$ are admissible. Then we have
\[[\eta_{U_r,t,t'},\eta^{-1}_{U_l,s',s}] = \eta_{U, s, t}.\]

\end{lemma}
\begin{proof}
First, observe that by Lemma~\ref{L: marginals and cylinders}
\[s(U_l)\cap t^{-1}(U_r)=C_{I_l, v}\cap C_{I_r, v}=C_{I,v}=U,\]
where in the second equality we used Lemma~\ref{L: marginals and cylinders}.

Second observe that Lemma~\ref{L: disjoint cylinders} together with the assumption that the length of $I$ is at least $2n_0+1$ (made in Definition~\ref{D: convenient admissible}) implies that all other pairs in the list $s'^{-1}(U_l), U_l, s(U_l), t^{-1}(U_r),U_r, t'(U_r)$ are disjoint.

Third and last, observe that if $(W,s',s),(V,t,t')$ are any two admissible triplets such that $sW\cap t^{-1}V\neq \varnothing$ and such that all other pairs in the list  $s'^{-1}(W), W, s(W), t^{-1}(V),V, t'(V)$ are disjoint, then we have the identity
\[[\eta_{V,t,t'},\eta^{-1}_{W,s',s}] = \eta_{tW\cap s'^{-1}V,s,t}.\]
The last identity applied to $W=U_l, V=U_r$ gives the desired conclusion.
\end{proof}

\begin{proof}[Proof of Theorem \ref{T: finitely generated}]
Consider the set $T=\{\eta_{V, s,t}\}$, where $(V, s,t)$ runs over convenient admissible triplets such that $V$ is a Gray-code cylinder of length $2n_0+1$. We shall show that $T$ generates $[[M]]_t'$.

Let us first show that $\langle T \rangle $ contains all elements $\eta_{U, s, t}$ where $(U, s, t)$ is any convenient admissible triplet.  Let $(I, v)$ be the Gray code piece corresponding to $U$. We prove the claim by induction on $2n+1$, the length of $I$. Since $(I_l, s^{-1}v)$ and $(I_r, t(v))$ are centered of length $2n-1$, by the inductive assumption we have that $\eta_{U_l,s', s}$ and $\eta_{U_r, t, t'}$ belong to $\langle T\rangle$, where $U_l, U_r, s', t'$ are as in Lemma~\ref{L: commutator trick}. By Lemma~\ref{L: commutator trick} also $\eta_{U, s,t}\in \langle T \rangle$.

 To conclude the proof, by Lemma  \ref{L: first reduction} it is enough to show that every element of the form $\eta_{U, s, t}$, with $U$ clopen and $s,t\in S$,  lies in $\langle T\rangle$ (here the triplet $(U,s,t)$ is admissible, but not necessarily convenient). By taking a partition of $U$ into central Gray code cylinders, we may assume that $U$ is a centered Gray code cylinder of depth $2n+1\geq 2n_0+1$. Let it correspond to $(I, v)$. If $s^{-1}(v), v, tv$ project to three consecutive points on the Gray code line, then the triplet is convenient and we are done. The other cases are covered by taking suitable conjugates of convenient admissible triplets (the reason why we used the generating set $S^2$ instead of $S$ in the definition of a convenient admissible triplet is that this gives enough room to perform this step).
 Consider the case where $v, tv$ project to the same point on the Gray code line. Pick $t'\in S$ so that the points $s^{-1}(v), v, t'(v)$ project to three distinct consecutive points (hence, the triplet $U, s, t'$ is convenient admissible). Choose also $s'\in S^2$ so that $s'^{-1}(v)\neq s^{-1}(v)$ and $s'^{-1}(v), s^{-1}(v)$  project to the same point on the line (note that it may not be possible if we considered the generating set $S$  only). Let $V$ be the cylinder corresponding to $(I, t(v))$. Then the triplet $V, s', t't^{-1}$ is convenient admissible, and we have
 \[\eta_{U, s, t}= \eta_{V, s', t't^{-1}}^{-1}\eta_{U, s, t'}\eta_{V, s', t't^{-1}}.\]
 The other cases are done similarly.
  \qedhere 

\end{proof}
\subsection{Proof of Proposition~\ref{P: marginals}}
\label{S: proof of prop}

\begin{defin}
Let $(I, \gamma)\subset \Gamma$ be a central Gray code piece. We say that $(I, \gamma)$ \emph{branches at the left} if the left marginal $(I_l, \gamma)$ contains vertices that project to $\overline{I}_c$ but that do not belong to $I_c$. The connected components of $p^{-1}(\overline I_c)\cap I_l$ are called the \emph{left branches}. The left branch containing $\gamma$ coincides with $I_c$ and is called the \emph{core branch}.
 In the same way we define branching at the right. We say that $(I, \gamma)$ \emph{bi-branches} if it branches both at the left and at the right. 

\end{defin}
We begin with the following special case of Proposition~\ref{P: marginals}.
\begin{lemma}
Let $(I, \gamma)\subset \Gamma$ be a central Gray code piece. If $(I, \gamma)$ does not bi-branch, then its isomorphism class is uniquely determined by the isomorphism class of its marginals.
\end{lemma}

\begin{proof}
Let $(I_l, \gamma)$ and $(I_r, \gamma)$ be the marginals. Since the position of $\gamma$ is given, the identification between vertices of the marginals is uniquely determined on the core branch. So if one of the marginals does not have any other branches, this determines the isomorphism class of $(I, \gamma)$ completely.
\end{proof}

To conclude, we need to understand in which situations bi-branching may occur.

\begin{lemma}\label{L: branching and visible}
Let $(I, \gamma)\subset \Gamma$ be a Gray code piece. Then $(I, \gamma)$ branches at the left if and only if there is a bit which is $B$-visible in $\overline{I}_l\setminus \overline{I}_c$, is not visible in $\overline{I_c}$ and is non-zero in at least a point of $\overline{I}_c$. The same characterization holds for branching at right.
\end{lemma}
\begin{proof}
This is an elementary consequence of the definitions and of Remark~\ref{R: visible}. \qedhere

\end{proof}
 Let $\overline{I}\subset \overline{\Gamma}$. We say that a point $\overline{\gamma}\in \overline{I}$ is a \emph{root} for $\overline{I}$ if the position of its first non-zero bit, say $j$, is maximal among all points of $\overline{I}$. 
We further say that a point of $\overline{I}$ is an \emph{anti-root} if the position of its first non-zero bit is exactly $j-1$. This terminology is inspired by \cite{Amir-Virag:positivespeed}.

 \begin{remark}\label{R: root}
 It follows from the definition of the Gray code line $\overline{\Gamma}$ that every connected segment has at least one and at most two roots, and if there are two distinct ones, they are neighbours (this follows from the fact that between any two non-neighboring points of $\oGamma$ beginning with a sequence of $j$ zeros there is at least a point beginning with a longer sequence of zeros).
 It could be that $\overline{I}$ has no anti-roots.
\end{remark}

\begin{lemma}\label{L: preimage root}
Let $(I, \gamma)$ be a Gray code piece, and $\overline{\rho}\in \overline{I}$ be a root. Then $p^{-1}(\overline{\rho})\cap I$ is a connected sub-graph of $I$.
\end{lemma}

\begin{proof}
Let $j$ be the position of the first '$\ast$' bit in $\overline{\rho}$. We claim that any two sequences $\rho, \rho'\in p^{-1}(\overline{\rho})\cap I$ can only differ in the $j$th bit, and in the $j+1$th bit if the latter is non-zero.  Assume that they differ in the $r$th bit with $r>j+1$. Then this bit must be visible in $\overline{I}$. This implies that there is a sequence in $\overline{I}$ whose first $'\ast'$ bit is at position $>j$, contradicting the definition of a root. This is enough, since then $\rho$ and $\rho'$ are connected by an edge corresponding to a generator in $B$.

\end{proof}

\begin{defin}
Let $(I, \gamma)\subset \Gamma$ be a central Gray code piece with $\overline{I}=[\overline{\gamma}_{-n}, \cdots, \overline{\gamma}_0, \cdots \gamma_n]$. We say that $(I, \gamma)$ is a \emph{quasi-level} if the roots of $\overline{I}$  are contained in the two leftmost vertices $\{\overline{\gamma}_{-n}, \overline{\gamma}_{-n+1}\}$, its antiroots are contained in the two rightmost vertices  $\{\overline{\gamma}_{n-1}, \overline{\gamma}_{n}\}$, and there is at least one anti-root, or if the symmetric situation (exchanging left and right) occurs.

A quasi-level has \emph{depth} $j$, where $j$ is the position of the first non-zero bit in a root.

\end{defin}

The reason for the name is that a quasi-level is essentially isomorphic the the finite Schreier graph for the action of $M$ on the $j$-th level of the tree, up to a bounded number of edges close to the left and right extremities.
\begin{lemma}
If $(I, \gamma)$ bi-branches, then it is a quasi-level.
\end{lemma}
\begin{proof}
Assume that $(I, \gamma)$ bi-branches. Let $\overline{\rho}\in \overline{I}$ be a root.  By Lemma~\ref{L: preimage root} and the fact that $(I, \gamma)$ bi-branches, we conclude that $\overline{\rho}\notin \overline{I}_c$. Hence, by Remark~\ref{R: root} and up to exchanging left and right, we may suppose that all  roots are contained in $\overline{I}_l\setminus \overline{I}_c$. 
We need to check that there is at least one anti-root and that anti-roots belong to $\overline{I}_r\setminus \overline{I}_c$.
Let $j$ be the position of the first $\ast$ in a root.
By Lemma~\ref{L: branching and visible}, there is at least a visible bit in $\overline{I}_r\setminus \overline{I}_c$ which is not visible in $\overline{I}_c$ and is $\ast$ in at least one point of $\overline{I}_c$. We claim that such a bit is necessarily at position $j$. Let $r$ be its position. Assume that $r>j$. Then it is preceded either by the prefix $0^{r-1}$ or by the prefix $0^{r-2}\ast$. In both cases the first $\ast$-bit in the corresponding sequence is at position $\geq j$, contradicting the fact that the roots are contained in the two left-most vertices. Assume that $r<j$. Since the bit at position $r$ is zero in the root and $\ast$ in a point in $\overline{I}_c$, there is an edge in $\overline{I}_c$ where it switches from $0$ to $\ast$. But then at this point it is also be visible, contradiction.
It follows that there is a point in $\overline{I}_r\setminus\overline{I}_c$ that has a visible bit at position $j$. Since this point is not a root, it also has a $\ast$ at position $j-1$. It follows that it is an anti-root.
If there was another anti-root within $\overline{I}_c$, the bit at position $j$ would be visible in $\overline{I}_c$, contradicting the previous reasoning. \qedhere
 \end{proof}

\begin{lemma}
Let $(I, \Gamma)$ be a quasi-level. Then it is uniquely determined by its marginals.

\end{lemma}
\begin{proof}
Let $\overline{\alpha}\in \overline{I}_l\setminus\overline{I}_c$ be the root closest to the center of $\overline{I}$, and $\overline{\omega}\in \overline{I}_r\setminus \overline{I}_c$ be the anti-root closest to the center. Then the preimages $p^{-1}(\overline{\alpha})$ and $p^{-1}(\overline{\beta})$ in $I_l$ and $I_r$ are connected graphs, and can be recognized by looking at the isomorphism classes of $I_l$ and $I_r$ only. By looking at these finite graphs and at their positions it is possible to recognize the letters at position $j$ and $j+1$ in $\gamma$ and the exact value of $j$. This information is enough to reconstruct the quasi-level.
\end{proof}

\bibliographystyle{alpha}
\bibliography{these1-12-15}

\begin{thebibliography}{JMBMdlS15}

\bibitem[AAMBV16]{Liouvilletrees}
Gideon Amir, Omer Angel, Nicol\'as Matte~Bon, and B\'alint Vir\'ag.
\newblock The {L}iouville property for groups acting on rooted trees.
\newblock {\em Annales de l'IHP}, 2016.
\newblock To appear.

\bibitem[AAV13]{Amir-Angel-Virag:linear}
Gideon Amir, Omer Angel, and B{\'a}lint Vir{\'a}g.
\newblock Amenability of linear-activity automaton groups.
\newblock {\em J. Eur. Math. Soc. (JEMS)}, 15(3):705--730, 2013.

\bibitem[AV12]{Amir-Virag:speedexponent}
Gideon Amir and B{\'a}lint Vir{\'a}g.
\newblock Speed exponents of random walks on groups.
\newblock 2012.
\newblock {P}reprint, arXiv:1203.6226.

\bibitem[AV14]{Amir-Virag:positivespeed}
Gideon Amir and B{\'a}lint Vir{\'a}g.
\newblock Positive speed for high-degree automaton groups.
\newblock {\em Groups Geom. Dyn.}, 8(1):23--38, 2014.

\bibitem[BKN10]{Bartholdi-Kaimanovich-Nekrashevych}
Laurent Bartholdi, Vadim~A. Kaimanovich, and Volodymyr~V. Nekrashevych.
\newblock On amenability of automata groups.
\newblock {\em Duke Math. J.}, 154(3):575--598, 2010.

\bibitem[BM08]{Bezuglyi-Medynets:flipconjugacy}
S.~Bezuglyi and K.~Medynets.
\newblock Full groups, flip conjugacy, and orbit equivalence of {C}antor
  minimal systems.
\newblock {\em Colloq. Math.}, 110(2):409--429, 2008.

\bibitem[Bri14]{Brieussel:folner}
J{\'e}r{\'e}mie Brieussel.
\newblock Folner sets of alternate directed groups.
\newblock {\em Ann. Inst. Fourier (Grenoble)}, 64(3):1109--1130, 2014.

\bibitem[BV05]{Bartholdi-Virag:amenability}
Laurent Bartholdi and B{\'a}lint Vir{\'a}g.
\newblock Amenability via random walks.
\newblock {\em Duke Math. J.}, 130(1):39--56, 2005.

\bibitem[Cor14]{Cornulier:Bourbaki}
Yves Cornulier.
\newblock {G}roupes pleins-topologiques [d'apr\`es {M}atui, {J}uschenko,
  {M}onod,...].
\newblock {\em Ast\'erisque}, (361):Exp. No. 1064, 2014.
\newblock S{\'e}minaire Bourbaki. Vol. 2012/2013.

\bibitem[DDMN10]{SchreierBasilica}
Daniele D'Angeli, Alfredo Donno, Michel Matter, and Tatiana Nagnibeda.
\newblock Schreier graphs of the {B}asilica group.
\newblock {\em J. Mod. Dyn.}, 4(1):167--205, 2010.

\bibitem[GPS99]{Giordano-Putnam-Skau:flipconjugacy}
Thierry Giordano, Ian~F. Putnam, and Christian~F. Skau.
\newblock Full groups of {C}antor minimal systems.
\newblock {\em Israel J. Math.}, 111:285--320, 1999.

\bibitem[GPS04]{Giordano-Putnam-Skau:affable}
Thierry Giordano, Ian Putnam, and Christian Skau.
\newblock Affable equivalence relations and orbit structure of {C}antor
  dynamical systems.
\newblock {\em Ergodic Theory Dynam. Systems}, 24(2):441--475, 2004.

\bibitem[Gri84]{Grigorchuk:grigorchukgroups}
R.~I. Grigorchuk.
\newblock Degrees of growth of finitely generated groups and the theory of
  invariant means.
\newblock {\em Izv. Akad. Nauk SSSR Ser. Mat.}, 48(5):939--985, 1984.

\bibitem[GS83]{Gupta-Sidki}
Narain Gupta and Sa{\"{\i}}d Sidki.
\newblock On the {B}urnside problem for periodic groups.
\newblock {\em Math. Z.}, 182(3):385--388, 1983.

\bibitem[GW15]{Glasner-Weiss:URS}
Eli Glasner and Benjamin Weiss.
\newblock Uniformly recurrent subgroups.
\newblock In {\em Recent trends in ergodic theory and dynamical systems},
  volume 631 of {\em Contemp. Math.}, pages 63--75. Amer. Math. Soc.,
  Providence, RI, 2015.

\bibitem[G{\.Z}02]{Grigorchuk-Zuk:basilica}
Rostislav~I. Grigorchuk and Andrzej {\.Z}uk.
\newblock On a torsion-free weakly branch group defined by a three state
  automaton.
\newblock {\em Internat. J. Algebra Comput.}, 12(1-2):223--246, 2002.
\newblock International Conference on Geometric and Combinatorial Methods in
  Group Theory and Semigroup Theory (Lincoln, NE, 2000).

\bibitem[JM13]{Juschenko-Monod:simpleamenable}
Kate Juschenko and Nicolas Monod.
\newblock Cantor systems, piecewise translations and simple amenable groups.
\newblock {\em Ann. of Math. (2)}, 178(2):775--787, 2013.

\bibitem[JMBMdlS15]{ExtAmen}
Kate Juschenko, Nicol{\'a}s Matte~Bon, Nicolas Monod, and Mikael de~la Salle.
\newblock Extensive amenability and an application to interval exchanges.
\newblock 2015.
\newblock {P}reprint, arXiv:1503.04977.

\bibitem[JNdlS13]{Juschenko-Nekrashevych-Salle:recurrentgroupoids}
Kate Juschenko, Volodymyr Nekrashevych, and Mikael de~la Salle.
\newblock Extensions of amenable groups by recurrent groupoids.
\newblock 2013.
\newblock Preprint, arXiv:1305.2637v2.

\bibitem[Jus15]{Juschenko:nonelementary}
Kate Juschenko.
\newblock Non-elementary amenable subgroups of automata groups.
\newblock 2015.
\newblock Preprint, arXiv:1504.00610.

\bibitem[Mat06]{Matui:simple}
Hiroki Matui.
\newblock Some remarks on topological full groups of {C}antor minimal systems.
\newblock {\em Internat. J. Math.}, 17(2):231--251, 2006.

\bibitem[Mat12]{Matui:groupoids}
Hiroki Matui.
\newblock Homology and topological full groups of \'etale groupoids on totally
  disconnected spaces.
\newblock {\em Proc. Lond. Math. Soc. (3)}, 104(1):27--56, 2012.

\bibitem[MB15]{GrigorchukTopFull}
Nicol{\'a}s Matte~Bon.
\newblock Topological full groups of minimal subshifts with subgroups of
  intermediate growth.
\newblock {\em J. Mod. Dyn.}, 9(01):67--80, 2015.

\bibitem[Nek05]{Nekrashevych:book}
Volodymyr Nekrashevych.
\newblock {\em Self-similar groups}, volume 117 of {\em Mathematical Surveys
  and Monographs}.
\newblock American Mathematical Society, Providence, RI, 2005.

\bibitem[Nek06]{Nekrashevych:inversesemigroups}
Volodymyr Nekrashevych.
\newblock Self-similar inverse semigroups and {S}male spaces.
\newblock {\em Internat. J. Algebra Comput.}, 16(5):849--874, 2006.

\bibitem[Nek15]{Nekrashevych:groupoidsimple}
Volodymyr Nekrashevych.
\newblock Simple groups of dynamical origin.
\newblock 2015.
\newblock Preprint, arXiv:1601.01033.

\bibitem[Nek16]{Nekrashevych:periodic}
Volodymyr Nekrashevych.
\newblock Periodic groups from minimal actions of the infinite dihedral group.
\newblock 2016.
\newblock Preprint, arXiv:1501.00722.

\bibitem[Vor12]{Vorobets:Schreier}
Yaroslav Vorobets.
\newblock Notes on the {S}chreier graphs of the {G}rigorchuk group.
\newblock In {\em Dynamical systems and group actions}, volume 567 of {\em
  Contemp. Math.}, pages 221--248. Amer. Math. Soc., Providence, RI, 2012.

\end{thebibliography}
\end{document}